\documentclass[12pt]{amsart} 
 \usepackage{amssymb, latexsym, amsmath} 

{\normalsize }

\begin{document} 

 \theoremstyle{plain} 
 \newtheorem{theorem}{Theorem}[section] 
 \newtheorem{lemma}[theorem]{Lemma} 
 \newtheorem{corollary}[theorem]{Corollary} 
 \newtheorem{proposition}[theorem]{Proposition} 

\theoremstyle{definition} 
\newtheorem{definition}{Definition}
\newtheorem{example}[theorem]{Example}
\newtheorem{remark}[theorem]{Remark}

\title[Galois correspondence]{On the Galois correspondence for Hopf Galois structures arising from finite radical algebras and Zappa-Sz\'ep products}

\author{Lindsay N. Childs}
\address{Department of Mathematics and Statistics\\
University at Albany\\
Albany, NY 12222}
\email{lchilds@albany.edu}

\date{\today}

\newcommand{\QQ}{\mathbb{Q}} 
\newcommand{\FF}{\mathbb{F}} 
\newcommand{\ZZ}{\mathbb{Z}}
\newcommand{\ZZm}{\mathbb{Z}/m\mathbb{Z}}
\newcommand{\ZZp}{\mathbb{Z}/p\mathbb{Z}}
\newcommand{\ben}{\begin{eqnarray}} 
\newcommand{\een}{\end{eqnarray}} 
\newcommand{\dis}{\displaystyle} 
\newcommand{\beal}{\[ \begin{aligned}} 
\newcommand{\eeal}{ \end{aligned} \]} 
\newcommand{\lb}{\lambda} 
\newcommand{\Gm}{\Gamma} 
\newcommand{\gm}{\gamma} 
\newcommand{\bpm}{\begin{pmatrix}} 
\newcommand{\epm}{\end{pmatrix}} 
\newcommand{\Fp}{\mathbb{F}_p} 
\newcommand{\Fpx}{\mathbb{F}_p^{\times}} 
\newcommand{\Ann}{\text{Ann} }
\newcommand{\tb}{\textbullet \ }
\newcommand{\GL}{\mathrm{GL}}

\newcommand{\M}{\mathrm{M}}
\newcommand{\Aut}{\mathrm{Aut}}
\newcommand{\End}{\mathrm{End}}
\newcommand{\mbf}{\mathbf}
\newcommand{\Perm}{\mathrm{Perm}}
\newcommand{\PGL}{\mathrm{PGL}}
\newcommand{\diag}{\mathrm{diag}}
\newcommand{\Reg}{\mathrm{Reg}}
\newcommand{\Hol}{\mathrm{Hol}}
\newcommand{\Inn}{\mathrm{Inn}}
\newcommand{\InHol}{\mathrm{InHol}}
\newcommand{\Lc}{\mathcal{L}}
\newcommand{\Hom}{\mathrm{Hom}}
\newcommand{\ol}{\overline}
\newcommand{\olx}{\overline{x}}

 \begin{abstract}
  Let $L/K$ be a $G$-Galois extension of fields with an $H$-Hopf Galois structure of type $N$.   We study the Galois correspondence ratio $GC(G, N)$,  which is the proportion of intermediate fields $E$ with $K \subseteq E \subseteq L$ that are in the image of the Galois correspondence for the $H$-Hopf Galois structure on $L/K$.   The Galois correspondence ratio for a Hopf Galois structure can be found by translating the problem into counting certain subgroups of the corresponding skew brace.  We look at skew braces arising from finite radical algebras $A$ and from Zappa-Sz\'ep products of finite groups, and in particular when $A^3 = 0$ or the Zappa-Sz\'ep product is a semidirect product, in which cases the corresponding skew brace is a bi-skew brace, that is,  a set $G$ with two group operations $\circ$ and $\star$ in such a way that $G$ is a skew brace with either group structure acting as the additive group of the skew brace.    We obtain the Galois correspondence ratio for  several examples.  In particular, if $(G, \circ, \star)$ is a bi-skew brace of squarefree order $2m$ where $(G, \circ) \cong Z_{2m}$ is cyclic and $(G, \star) \cong D_m$ is dihedral,  then for large $m$, $GC(Z_{2m},D_m), $ is close to 1/2 while $GC(D_m, Z_{2m})$ is near 0.
 
\end{abstract}
\maketitle

\

\

\section{Introduction}

 In 1969 S. Chase and M. Sweedler [CS69] defined the notion of an $H$-Hopf Galois structure on a finite extension of commutative rings.  The notion in particular applies to the case of a separable field extension $L/K$, where it extends the classical notion of a Galois extension, the case where  $L/K$ is normal, $G$ is the Galois group of $L/K$, and $H$ is the group ring $KG$.  

Let $L/K$ be a $G$-Galois extension of fields.  Let $H$ be a cocommutative $K$-Hopf algebra that acts on $L$ as an $H$-module algebra, and suppose $L$ is an $H$-Hopf Galois extension of $K$.  Then, as shown in [GP87], $L \otimes_K H = LN$ for some regular subgroup $N$ of $\Perm(G)$, where $N$ is normalized by the image $\lambda(G)$ of the left regular representation $\lambda:  G \to \Perm(G)$.    In turn, if $N$ is a regular subgroup of $\Perm(G)$ normalized by $\lambda(G)$, then $N$ yields by Galois descent a $K$-Hopf algebra $H = L[N]^G$ which acts on $L/K$ making $L/K$ a $H$-Hopf Galois extension.  In this way there is a bijection between Hopf Galois structures on the $G$-Galois extension $L/K$ and regular subgroups of $\Perm(G)$ normalized by $\lambda(G)$.  

Given an $H$-Hopf Galois structure on $L/K$, there is a Galois correspondence, originally described by Chase and Sweedler [CS69], namely, an injective correspondence from $K$-subHopf algebras of $H$ to fields $E$ with $K \subset E \subset L$, by
\[ H' \mapsto L^{H'} = \{x \in L | h'x = \epsilon(h')x \text{ for all $h'$ in }H'\},\]
 where $\epsilon: H \to K$ is the counit map.  But in contrast to classical Galois theory, the Galois correspondence for a Hopf Galois extension can fail to be surjective.     The first set of examples where surjectivity fails is in [GP87]:  let $N$ be the regular subgroup $N = \lambda(G)$ of $\Perm(G)$.  Then $N$ is normalized by itself, and the subgroups $N'$ of $N = \lambda(G)$ that are normalized by $\lambda(G)$ are the subgroups $\lambda(G')$ where $G'$ is a normal subgroup of $G$. So if $G$ is non-abelian and has non-normal subgroups, then surjectivity fails.  (The classical Galois structure on $L/K$  given by the Galois group corresponds to $\rho(G)$, the image in $\Perm(G)$ of the right regular representation of $G$ in $\Perm(G)$, $\rho(g)(h) = hg^{-1}$ for $g, h$ in $G$:  $\rho(G)$, hence every subgroup of $\rho(G)$, is centralized, hence normalized by $\lambda(G)$.)

But except for the Greither-Pareigis examples, very little was known about the image of the Galois correspondence for an $H$-Hopf Galois structure on a $G$-Galois extension $L/K$ of fields until [CRV16]  observed that the $K$-subHopf algebras of $H$ correspond bijectively to the subgroups $N'$ of  $N$ that are normalized by $\lambda(G)$.  Using that result, [Ch17] examined $G$-Galois extensions $L/K$ where $G$ is an abelian $p$-group of order $p^n$ and $L/K$ has an $H$-Hopf Galois structure of type $N$, also an abelian $p$-group of order $p^n$.   If $H$ has type $N$, then there is a regular  embedding of $G$ into $\Hol(N)$:  call the image $T$.  Then ([CDVS06], [FCC12]), that regular subgroup $T$ defines a commutative nilpotent ring structure $A = (N, +, \cdot)$ on the additive abelian group $N$ so that the regular subgroup $T$, and hence the Galois group $G$, is isomorphic to the adjoint group $(A, \circ)$ of the nilpotent ring $A$.  In that setting,  for the Hopf Galois structure given by $H$, the image of the Galois correspondence is in bijective correspondence with the ideals of $A$.  

If $A^e = 0$ and $e < p$, hence $G$ and $N$ are elementary abelian $p$-groups,  [CG18] computed upper and lower bounds on the proportion of  subgroups of $A$ that are ideals, and showed, for example, that if $4 \le e < p$, then the proportion of subgroups that are ideals is $< .01$. 

We note that the only known cases of  a non-classical $H$-Hopf Galois structure on a $G$-Galois field extension where the Galois correspondence for $H$ is surjective are for $G$ a non-abelian Hamiltonian group (every subgroup is normal) where $H$ corresponds to $\lambda(G)$, or for $G$ cyclic of odd prime power order and $H$ any Hopf Galois structure ([Ch17], Proposition 4.3).

The remainder of the paper is organized as follows.  Section 2 describes the relationship between Hopf Galois structures and skew braces, and section 3 describes the Galois correspondence ratio in the skew brace setting.  Section 4 looks at that ratio for skew braces arising from radical algebras, illustrated by a four-dimensional example $A = A^0_{4, 25}$.  Section 5 introduced bi-skew braces and finds the Galois correspondence ratio for the other skew brace structure for $A$, which exists because $A^3 = 0$. Section 6 looks at the Galois correspondence ratios  arising from fixed point free pairs of homomorphisms to a Zappa-Sz\'ep product of two finite groups from the corresponding direct product, with an example.  Section 7 specializes to semi-direct products of groups, which yield bi-skew braces, hence two Galois correspondence ratios.  Section 8 looks at examples where the order of $G$ is square-free, and in particular determines the two highly divergent ratios for $G$ a class of generalized dihedral groups.

Many thanks to the referees for their numerous helpful comments and corrections on a previous version of this paper.

\section{Hopf Galois structures and skew braces}
The concept of skew brace was first defined in [GV17] and the connection of skew braces with Hopf Galois structures was described in [SV18].  Following [SV18], in what follows, ``skew brace'' and ``brace'' will always mean ``skew left brace'' and ``left brace'', respectively.  

Finite radical rings are skew braces, so it was natural to generalize the description of the Galois correspondence ratio for radical rings in [Ch17] to the setting of skew braces.  This was done in [Ch18]. 

To see how this works, we first reexamine the description of  a Hopf Galois structure.

Let $G$ be a finite group, and denote the operation on $G$ by $\circ$.  Let $L/K$ be a $(G, \circ)$-Galois extension of fields.   If $L$ is also an $H$-Hopf Galois extension of $K$,  then  $L \otimes_K H = LN$ for some regular subgroup $N$ of $\Perm(G)$, where $N$ is normalized by the image $\lambda_\circ(G)$ of the left regular representation $\lambda_\circ:  G \to \Perm(G)$.
 
Since $N$ is a regular subgroup of $\Perm(G)$, the map $b: N \to G$ given by $n \mapsto n(e)$ is a bijection.  Then $b$ defines an operation $\star$ on $G$, as follows:  if $b(n_1) = g_1, b(n_2) = g_2$, we define
\[ g_1 \star g_2 = b(n_1n_2). \] 
Then  $N = \lambda_\star(G)$, the image of the left regular representation map $\lambda_\star$ corresponding to the operation $\star$.  For setting
$b(n_1) = g_1, b(n_2) = g_2$ and  $g_1\star g_2 = b(n_1n_2)$, then the action of $N \subset \Perm(G)$ on $G$ is
\beal n_1(g_2) &= n_1(n_2(e)) = (n_1n_2)(e)\\
& = b(n_1n_2) = g_1 \star g_2 = \lambda_{\star}(g_1)(g_2).\eeal

If $N = \lambda_\star(G)$ in $\Perm(G)$ is normalized by $\lambda_\circ(G)$,  then $\lambda_\circ(G)$ is contained in $\Hol(G, \star)$, the normalizer of $\lambda_\star(G)$ in $\Perm(G)$. This observation connects Hopf Galois structures with skew braces.
 
\begin{definition}  A  skew  brace is a finite set $B$ with two operations, $\star$ and $\circ$, so that $(B, \star)$ is a group (the ``additive group''), $(B, \circ)$ is a group, and the compatibility condition
\[ a \circ (b \star c) = (a \circ b) \star a^{-1} \star (a \circ c) \]
holds for all $a, b, c$ in $B$.  Here $a^{-1}$ is the inverse of $a$ in $(B, \star)$. Denote the inverse of $a$ in $(B, \circ)$ by $\overline{a}$.  \end{definition}

If $B$ has two operations $\star$ and $\circ$ and is a skew brace with $(B, \star)$ the additive group, then we write $B = B(\circ, \star)$  (i. e. the additive group operation is on the right).

A brace is a skew brace with abelian additive group.  Every brace $(A, \circ, +)$ with abelian circle group is  a radical algebra $(A, +, \cdot)$ [Ru07], where the algebra multiplication on $A$ is defined by $a \cdot b = a\circ b - a - b$.  

A set $B$ with two group operations $ \circ$ and $\star$  has two left regular representation maps:
 \beal \lb_\star:  &B\to \Perm(B), \lb_\star(b)(x) = b \star x,\\
 \lb_{\circ}: & B \to \Perm(B), \lb_{\circ}(b)(x) = b \circ x. \eeal
 Then  Guarnieri and Vendramin proved ([GV17], Proposition 1.9):
 
\begin{theorem}\label{1.9}  A set $(B, \circ, \star)$ with two group operations is a skew brace if and only if  the group homomorphism  $\lb_\circ:  (B , \circ) \to \Perm(B)$ has image in
\[\Hol(B, \star) = \lambda_\star(B) \Aut(B, \star) \subset \Perm(B),\] the normalizer in $\Perm(B)$ of $\lambda_{\star}(G)$.  \end{theorem}

Let $L/K$ be a Galois extension with group $G = (G, \circ)$.  Let $H$ be a $K$-Hopf algebra giving a Hopf Galois structure of type $N$ on $L/K$.  Using $N$ to define the group structure $(G, \star)$ on $G$  and then identifying $N$ with $(G, \star)$ as above, then $\lambda_\circ(G)$ normalizes $N = (G, \star)$, so $\lambda_\circ$ is contained in $\Hol(G, \star)$.  Thus $(G, \circ, \star)$ is a skew brace.  

Conversely, let $(G, \circ, \star)$ be a skew brace.  Let $L/K$ be a Galois extension with Galois group $(G, \circ)$.   Then $L/K$ has a Hopf Galois structure of type $(G, \star)$.  For  given the skew brace structure $(G, \circ, \star)$ on the Galois group $(G, \circ)$ of $L/K$, then  $\lambda_\circ(G)$ is contained in $\Hol(G, \star)$, and so the subgroup $N = \lambda_\star(G) \subset \Perm(G)$ is normalized by $\lambda_\circ (G)$.  So $N$ corresponds by Galois descent to a Hopf Galois structure on $L/K$ of type $(G, \star)$. 

\begin{remark} We note that the correspondence that connects regular subgroups $N$ of $\Perm(\Gamma)$ normalized by $\lambda(\Gamma)$ and isomorphic to $(G, \star)$ to isomorphism types of skew braces $(G, \circ, \star)$ with $(G, \circ) \cong \Gamma$ and $(G, \star) \cong N$ is not bijective.  We have (c. f. [NZ19], Corollary 2.4):

\begin{proposition} [Byott, Nejabati Zenouz] \label{BZ} Given an isomorphism type $(B, \circ, \star)$ of skew brace, the number of Hopf Galois structures on a Galois extension $L/K$ with Galois group isomorphic to $(B, \circ)$ and skew brace isomorphic to $(B, \circ, \star)$ is 
\[ \Aut(B, \circ)/\Aut_{sb}(B, \circ, \star).\]
\end{proposition}

Here $\Aut_{sb}(B, \circ, \star)$ is the group of skew brace automorphisms of $(B, \circ, \star)$, that is, maps from $B$ to $B$ that are simultaneously group automorphisms of $(B, \star)$ and of $(B, \circ)$. \end{remark}

\section{The Galois correspondence for skew braces}
Given that an $H$-Hopf Galois structure on a $G$-Galois extension corresponds to a skew brace $(G, \circ, \star)$ so that $G \cong (G, \circ)$ and $H$ has type $(G, \star)$, we can rephrase the question of identifying the $K$-sub-Hopf algebras of $H$, and hence the question of counting the size of the image of the Galois correspondence for $H$, into  a question of identifying and counting certain subgroups of the skew brace $(G, \circ, \star)$.  This was done in [Ch18].

Let $L/K$ be Galois with group $(G, \circ)$ and $H$- Hopf Galois where $H$ has type $(G, \star)$, so $(G, \circ, \star)$ is a skew brace.   We're interested in the Galois correspondence ratio,
\[GC((G, \circ), (G, \star)) =\frac{|\{\text{$E$ in the image of the Galois correspondence for $H$}\}|}{|\{E:  K \subset E \subset L\}|} .\]

The numerator counts  the $\lambda_{\circ}(G)$-invariant subgroups of $\lambda_{\star}(G)$.  Looking at them in the skew brace setting, we have

\begin{definition}  Let $(G, \circ, \star)$ be a skew brace.  A subgroup $(G', \star)$ of $(G, \star)$ is $\circ$-stable if $\lambda_\star (G')$ is closed under conjugation in $\Perm(G)$ by $\lambda_\circ(G)$.  \end{definition}

This condition is equivalent to 

\[  \text{For all } g \in G, g' \in G', (g \circ g')\star g^{-1} = h' \text{ is in }G'.\]

\begin{remark}  Lemma 2.2 of [DeC19] introduces the automorphism $\rho_g$ of $(G, \star)$ defined by 
\[ \rho_g(g') =  (g \circ g')\star g^{-1}.\]
So a $\circ$-stable subgroup $G'$ of $G$ is a subgroup invariant under the action of $\rho_g$ for all $g$.

Section 5 of [KT19] recasts the Galois correspondence for a Hopf Galois structure of type $(G, \star)$ on a $(G, \circ)$-Galois extension by replacing the skew brace $(G, \circ, \star)$ by its opposite skew brace $(G, \circ, \star')$.  Then a $\circ$-stable subgroup $G'$ of $G$ is what they call a quasi-ideal of the opposite brace.   A quasi-ideal $G'$ is an ideal of $G$ ([GV17], Definition 2.1) if and only if $G'$ is a normal subgroup of $(G, \circ)$.
\end{remark}

We observed in Proposition 4.1 of  [Ch18] that a $\circ$-stable subgroup of $(G, \circ, \star)$ is a subgroup of both $(G, \circ)$ and $(G, \star)$.

Thus, if $L/K$ is a $(G, \circ)$-Galois extension and an $H$-Galois extension where the $H$ structure corresponds to a skew brace structure $(G, \circ, \star)$ on $G$, so that $H$ has type $(G, \star)$, then  the Galois correspondence ratio becomes
\[GC( (G, \circ), (G, \star)) =\frac{|\{\text{$\circ$-stable subgroups of $(G, \star)$}\}|}{|\{\text{subgroups of $(G, \circ)$}\}|} .\]

\section{Radical algebras and the Galois correspondence for corresponding Hopf Galois structures}

A ring  $A = (A, +, \cdot)$ (without unit) is a  radical ring if the operation 
\[ a \circ b = a + b + a\cdot b\]
for all $a$, $b$ in $A$ makes $(A, \circ)$ into a group with identity 0.  If $A$ is a finite ring, then $A$ is Artinian (has descending chain condition on left (or right) ideals), so $A$ is nilpotent (every element of $A$ is nilpotent), c. f. Section 1.2 of [He61].  

Conversely, a nilpotent ring $A$ is a radical ring. That means, if we define the operation $\circ$ on $A$ by 
\[ a \circ b = a + b + ab,\]
then $(A, \circ)$ is a group.  Associativity is clear, the identity is 0, and the inverse of $a$ in $(A, \circ)$, denoted by $\overline{a}$, is
\[ \overline{a} = -a + a^2 - a^3 + \ldots,\]
a finite sum because $A$ is nilpotent.
(The circle group of $A$ is isomorphic (by $a \mapsto -a$) to what is sometimes called the adjoint group of $A$.) 

A radical ring $A$ is a brace $(A, \circ, +)$, as was first observed in [Ru07], and hence is a skew brace.   

Let $(A, +, \cdot)$ be a finite radical ring, and let $a \circ b = a + b + a\cdot b$.  In [Ch18] we observed: 

\begin{proposition}  If  $(A, \circ, +)$ is the skew brace arising from a radical ring $A$, then the $\circ$-stable subgroups of $(A, +)$  are the left ideals of the ring $A$.  \end{proposition}

For the reader's convenience, here is a proof.

\begin{proof} Suppose $(G', +)$ is a $\circ$-stable subgroup of $(A, +)$.  Then for all $g$ in $A$, $g'$ in $G'$, there is some $h'$ in $G'$ so that $g \circ g' = h' + g$, or equivalently:
\beal g + g' + gg' &= h' + g\\
gg' &= h' - g' \text{ in } G'.\eeal
So $G'$ is closed under left multiplication by elements of $A$, hence is a left ideal of $A$.  

Conversely, if $G'$ is a left ideal of $A$, then for all $g'$ in $G'$, $g$ in $A$, $gg'$ is in $G'$, so $gg' + g'  = h'$ is in $G'$.  So
\[g \circ g' = g + g' + gg' = h' + g,\]
the condition that $(G', +)$ is a $\circ$-stable subgroup of $(A, +)$. \end{proof}

Thus if $A$ is a finite radical ring and $L/K$ is a Galois extension with Galois group $(A, \circ)$ with a Hopf Galois structure with $H$ of type $(A, +)$, then the Galois correspondence ratio for $H$ acting on $L/K$ is
\[ GC( (A, \circ), (A, \star)) =  
\frac{|\{\text{left ideals of $A$}\}|}{|\{\text{subgroups of $(A, \circ)$}|} .\]

In [Ch18] we illustrated this result by looking at two non-commutative nilpotent $\Fp$-algebras  of dimension 3 of [DeG17].  Here is a four-dimensional example.

\begin{example}  We look at de Graaf's [DeG17] $\Fp$-algebra $A^0_{4, 21}$ with an $\Fp$-basis consisting of elements $a, b, c, d$ with multiplication given by $a^2 = c, ab = d$ and all other products of basis elements $= 0$.  Assume $p > 2$.  Then a subset $J$ of $A$ is a left ideal of $A$ if and only if $aJ \subset J$. So if $J$ is a left ideal of $A$ and $r = r_1a + r_2b + r_3c + r_4d$ is in $J$, then $ar = r_1c + r_2d$ is also in $J$.  

A 4 $\times$ 4 row-reduced echelon matrix in $M_4(\Fp)$ with  pivots (leading ones) in columns $c_1,\ldots, c_r$ will be called a matrix of form $(c_1 \ldots c_r)$.  Every non-zero subspace of $\Fp^4$ is generated by the non-zero rows of a unique row-reduced echelon matrix.  

From that viewpoint, we find that the non-zero left ideals of $A$ correspond to all echelon matrices of forms (3), (4), (34), (24), (134), (234), (1234) $= I$, and also matrices of form (13) with two parameters $r$ and $s$:   the non-zero rows of such a echelon matrix form the matrix
\[ \bpm 1 & r & 0 & s\\0&0&1&r \epm .\]
Counting the number of parameters for each form gives
\[ p + 1 + 1 + p + p + 1 + 1 + p^2 = p^2 + 3p + 4\]
non-zero left ideals of $A$.

Suppose $L/K$ is a Galois extension with Galois group $(A, \circ)$ and has a $H$-Hopf Galois structure of type $(A, +)$ corresponding to the skew brace $(A, \circ, +)$.  Then the Galois correspondence ratio for the $H$-Hopf Galois structure is 
\[GC((G, \circ), (G, +)) = \frac {|\{\text{ left ideals of $A$}\}|}{|\{\text{subgroups of $(A, \circ)$}\}|}. \]
Including the zero ideal, the numerator is $p^2 + 3p + 5$.  

To find the denominator, the number of subgroups of $(A, \circ)$, is a bit more challenging.  

Now $A$ is the $\Fp$-algebra with $\Fp$-basis $\{a, b, c, d\}$, where $a^2 = c, ab = d$ and all other products $= 0$.  For $x, y$ in $A$, define $x \circ y = x + y + xy$.  Since $p > 2$,  $(A, \circ)$ is a group of exponent $p$, for given any $r, s, t, u$ in $\Fp$,  $ra \cdot ra = r^2c, ra\cdot sb = rsd$, so one sees easily by induction that
\[ (ra + sb + tc + ud)^{\circ m} = mra + msb +  (m +\binom m2)r^2c +  (m + \binom m2)rsd\]
 for all $m \ge 1$.  

Since  $(\langle b, c, d\rangle, \circ) = (\langle b, c, d\rangle, +)$, there are $2p^2 +2p + 4$ subgroups of $\langle b, c, d \rangle$.  

Then we must count the subgroups of $(A, \circ)$ that contain an element of the form 
$\alpha = ra + sb + tc + wd$ with $r \ne 0$. By the formula above, we can assume $r = 1$.  Thus we have $p^3$ cyclic subgroups of order $p$.   We need also to count non-cyclic subgroups $G'$ that are not contained in $\langle b, c, d \rangle$.  

Let $w_1, w_2, \ldots, $ denote elements of $\langle c, d \rangle$.  
Any subgroup of $(A, \circ)$ that contains $\alpha_1 = a + sb + w_1$ also contains 
 \[ \alpha_1^{-1} = -a -sb -w_1 + c + sd, \]
 where $\alpha_1 \circ \alpha_1^{-1} = 0$.  
 
Let $G' = \langle \alpha_1, \alpha_2 \rangle$, where
\[ \alpha_1 = a + sb + w_1 \text{  and  }\alpha_2 = a + s'b + w_2. \]
First suppose $s' \ne s$.  Then $G'$ contains  
\[ \alpha_1^{-1} \circ \alpha_2 = (s'-s)b + w_3.\]
Now the $\circ$-subgroup $\langle b, c, d\rangle$ of $(A, \circ)$ is isomorphic to $\Fp^3$, so if  $(s'-s)b + w_3$ is in $G'$ and $s - s' \ne 0$, then  $G'$ contains $b + w_4$, with inverse $-b - w_4$, so contains
\[ (-s' b -w_5)\circ \alpha_2 = (-s'b -w_5) + a + s'b + w_2 = a + w_6,\]  
with $\circ$-inverse $-a -w_6+ c$.  So $G'$ contains $a + w_6$ and $b + w_4$.  But then $G'$ contains 
\beal &(-b - w_4) \circ (a + w_6) \circ (b + w_4) \circ (-a - w_6 + c)\\
&= (a - b -w_4+w_6)\circ (-a + b + w_4 - w_6 + c\\
&= c - c + d = d.\eeal
So if $G' = \langle \alpha_1,  \alpha_2  \rangle$, then 
\[ G' = \langle a + sc, b + tc, d \rangle\]
for $s, t$ in $\Fp$.  There are $p^2$ such groups.

If $s - s' = 0$, then $G' = \langle a + sb + tc + ud, t'c + u'd\rangle$ for some $t, t', u, u'$ in $\Fp$.    If $t' \ne 0$, then $G' = \langle a + sb + u''d, c + u''' d \rangle$ for $s, u'', u'''$ in $\Fp$,  so there are $p^3$ groups of that form.  If $t' = 0$, then $G' = \langle a + sb + tc, d\rangle$, so there are $p^2$ groups of that form.

Adjoining $c$ to any of the  groups of the last two forms yields $p$ groups of the form $G' = \langle a + sb, c, d \rangle$.    

Including $G' = G$, we have $2p^3 + 4p^2 + 3p + 5$ subgroups of $(G, \circ)$.

Thus the Galois correspondence ratio is 
\[GC((G, \circ), (G, +)) = \frac {p^2 + 3p + 5}{2p^3 + 4p^2 + 3p + 5}. \]
For large $p$ this is near $1/2p$.
\end{example}

\section{Bi-skew braces}

\begin{definition}  A bi-skew brace  is a finite set $B$ with two operations, $\star$ and $\circ$ so that $(B, \star)$ is a group, $(B, \circ)$ is a group, and $B$ is a skew brace with either group acting as the additive group:  that is, the two compatibility conditions
\[ a \circ (b \star c) = (a \circ b) \star a^{-1} \star (a \circ c) \]
and
\[ a \star (b \circ c) = (a \star b) \circ \ol{a} \circ  (a \star c) \]
hold for all $a, b, c$ in $B$.   \end{definition}

In Proposition 4.1 of  [Ch19] we showed that radical algebras $A$ with $A^3 = 0$ yield bi-skew braces.  That means that with $\circ$ defined as before, $(A, +, \circ)$ is a skew brace.  Thus if $L'/K'$ is a Galois extension with Galois group $(A, +)$, then $L'/K'$  has a Hopf Galois structure of type $(A, \circ)$ coming from the skew brace structure $(A, +, \circ)$ on $A$.   The image of the Galois correspondence for that Hopf Galois structure is then in bijective correspondence with the set of $+$-stable subgroups of $(A, \circ)$.   

A subgroup $J$ of $(A, \circ)$ is a $+$-stable subgroup if for all $g$ in $A$ and $h$ in $J$, there is some $h'$ in $J$ so that \[ g + h = h' \circ g .\]
Let $g, h$ be arbitrary elements of $J$ and note that $J$ is closed under $\circ$.  Then $J$ is also closed under $+$, hence is a subgroup of $(A, +)$.   Since $h' \circ g = h'g + h' + g$, an alternative version of $+$-stability is that for all $g$ in $A$, $h$ in $J$, there is $h'$ in $J
$ so that 
\[ g + h = h' \circ g = h' + g + h'g,\]
or  $h = h' + h'g$.

We show:  
\begin{proposition}  Let $(A, +, \circ)$ be the skew brace arising from the finite radical ring $(A, +, \cdot)$ with $A^3 = 0$.  Then the $+$-stable subgroups of $(A,\circ)$ are the right ideals of the ring $A$. \end{proposition}

\begin{proof}  Let $J$ be a right ideal of the radical ring $A$:  then $J$ is closed under addition and scalar multiplication on the right.  We show that $J$ is $+$-stable. Given $h$ in $J$, $g$ in $G$, let $h' = h - hg$.  Then $h'$ is in $J$ since $J$ is a right ideal, and since $A^3 = 0$, 
\[ h'g = (h - hg)g = hg  = h - h'.\]
So $J$ is $+$-stable.
  
Conversely, if $J$ is $+$-stable, then $J$ is an additive subgroup of $A$, and for all $h$ in $J$, $g$ in $A$, there is some $h'$ in $J$ so that $h'g = h - h'$. 
Since $A^3 = 0$, $0 = h'gg = hg - h'g$, so $hg = h'g = h - h'$.  Thus for all $h$ in $J$, $g$ in $G$, $hg$ is in $J$, and so $J$ is a right ideal of the radical ring $A$.\end{proof}

Thus for a $G = (A, +)$-Galois extension $L'/K'$ with an $H'$-Hopf Galois structure of type $(A, \circ)$, the Galois correspondence ratio is 
\[ GC((A. +), (A, \circ)) =  
\frac{|\{\text{right ideals of $A$}\}|}{|\{\text{subgroups of $(A, +)$}|} .\]

\begin{example}\label{8.2}
One motivation for this paper was to see if there might be any relationship between the Galois correspondence ratios for the two skew braces associated to a bi-skew brace.   

So we look again at de Graaf's [DeG17] example $A^0_{4, 21}$ with an $\Fp$-basis consisting of elements $a, b, c, d$ with multiplication given by $a^2 = c, ab = d$ and all other products of basis elements $= 0$.  Then a subset $J$ of $A$ is a  right ideal of $A$ if and only if  $Ja \subset J$ and $Jb \subset J$, hence  if $J$ is a right ideal of $A$ and $r = r_1a + r_2b + r_3c + r_4d$ is in $A$, then $ra^2 = r_1c$  and $rab = r_1d$ are in $J$.  Thus the non-zero right ideals of $A$ correspond to all echelon matrices of forms  (2), (3), (4), (23), (24), (34), (134), (234), (1234).  Counting the number of parameters for each form gives
\[ p^2 + p  + 1 + p^2 + p + 1 + p + 1 + 1 = 2p^2 + 3p + 4\]
non-zero right ideals of $A$.

To determine the proportions of intermediate fields that are in the image of the Galois correspondence for the  Hopf Galois structures corresponding to the skew brace $(A, +, \circ)$, we also need the numbers of subgroups of $(A, +)$. But since $(A, +) \cong \Fp^4$, the number of subgroups of $(A, +)$ is equal to the number of subspaces of $\Fp^4$, namely $p^4 + 3p^3 + 4p^2 + 3p + 5$.  

Thus if $L'/K'$ is a $(A, +)$-Galois extension with an $H'$-Hopf Galois structure of type $(A, \circ)$, then the proportion of subgroups of $(A, +)$ that are in the image of the Galois correspondence for $H$ is
\beal  GC((A, +), (A, \circ))
&= \frac{|\{\text{$+$-stable subgroups of $(A, \circ)$}\}|}{|\{\text{ subgroups of $(A, +)$}\}|}\\
&= \frac{|\{\text{right ideals of $A$}\}|}{|\{\text{ subgroups of $(A, +)$}\}|}\\
& = \frac{2p^2 + 3p + 5}{p^4 + 3p^3 + 4p^2 + 3p + 5}.\eeal
For large $p$ this is near $2/p^2$.

The two Hopf Galois extensions, related by the bi-skew brace arising from the radical algebra $A$ with $A^3 = 0$, have Galois correspondence ratios that behave like $1/2p$ and $2/p^2$ for large primes $p$.
\end{example}
 
\section{Zappa-Sz\'ep products and the Galois correspondence for corresponding Hopf Galois structures}

A finite group $G$ with identity $e$ and subgroups $G_L$ and $G_R$ is an internal Zappa-Sz\'ep product if $G = G_L G_R$ and $G_L \cap G_R = {e}$.  Other terminology:  two subgroups $G_L$ and $G_R$ of a finite group are complementary if $|G_L| \cdot |G_R| = |G|$ and $G_L \cap G_R = {e}$ ([By15], Section 7), or $G$ admits an exact factorization  through the subgroups $G_L$ and $G_R$ ([SV18], Example 3.6).  Thus every element $g$ of $G$ can be uniquely written as $g = g_L g_R$ for $g_L$ in $G_L$, $g_R$ in $G_R$.
  
Denote the group operation on $G$ by $\cdot$, usually omitted.

In general, for groups $\Gm, G$ of the same finite order, a fixed point free pair of homomorphisms from $\Gm$ to $G$ yields a Hopf Galois structure of type $G$ on a $\Gm$-Galois extension of fields, or equivalently, a skew brace structure $(G, \circ, \cdot)$ on the additive group $G = (G, \cdot)$, where $(G, \circ) \cong \Gm$.  In the Hopf Galois setting the method was first used for $\Gm = G$ in [CCo07] and in general in [BC12], c. f. Remark 7.2 of [By15].  In the skew brace setting the construction is noted without details in Example 3.6 of [SV18].  

Here is how it works for Zappa-Sz\'ep products.  

Given a Zappa-Sz\'ep product $G = G_L \cdot G_R$ there is a pair of homomorphisms $\beta_L$ and $\beta_R:  G_L \times G_R \to G$ given by $\beta_L(g_L, g_R) = g_L$, $\beta_R(g_L, g_R) = g_R$.  Since $G_L \cap G_R = \{e\}$, $(\beta_L, \beta_R)$ is a fixed point free pair of homomorphisms from $G_L \times G_R$ to $G$:
$\beta_L(g_L, g_R) = \beta_R(g_L, g_R)$ if and only if $g_L = g_R = e$.  

Using the left and right regular representations:   $\lambda, \rho:  G  \to \Perm(G)$ given by $\lambda(g)(x) = gx, \rho(g)(x) = xg^{-1}$ for $g, x$ in $G$,  the  fixed point free pair $(\beta_L, \beta_R)$   yields a regular embedding 
\[ \beta:  G_L \times G_R \to \lambda(G) \rtimes \Inn(G, \cdot) \subset \Hol(G, \cdot) \]
defined for $x$ in $G$ by
\beal \beta(g_L, g_R)(x) &=  \lambda(\beta_L(g_L, g_R)\rho(\beta_R(g_L, g_R)(x)\\
&= \lambda(g_L)\rho(g_R)(x) \\&= g_Lx g_R^{-1}\\
&= g_Lg_R^{-1}g_R x g_R^{-1}\\
&= \lambda(g_Lg_R^{-1})C(g_R)(x)  \subset \lambda(G)\cdot \Inn(G, \cdot).  \eeal
The regular embedding $\beta:  G_L \times G_R \to \Hol(G, \cdot)$ yields a bijection $b:  G_L \times G_R \to G$ by 
\[ b(g_L, g_R) = \beta(g_L, g_R)(e) = g_L e g_R^{-1} = g_Lg_R^{-1}.\]
Then $b$ defines a new group operation $\circ$ on $G$ from the direct product operation on $G_L \times G_R$ by 
\[ b(g_L, g_R) \circ b(h_L, h_R) = b((g_L, g_R)(h_L, h_R)) = b(g_Lh_L, g_Rh_R): \]
that is, 
\[ g_Lg_R^{-1} \circ h_Lh_R^{-1} = (g_Lh_L)(g_Rh_R)^{-1} = 
(g_Lh_L)(h_R^{-1}g_R^{-1})  = g_L(h_Lh_R^{-1})g_R^{-1},\]
or more concisely, $g_Lg_R^{-1} \circ h = g_L h g_R^{-1}$.  
Thus $b: G_L \times G_R \to (G, \circ)$ is an isomorphism, $\beta: G_L \times G_R \to \Hol(G, \cdot)$ becomes the left regular embedding $\lambda_\circ:  (G, \circ) \to \Hol(G, \cdot)$, and the circle operation on $G = G_LG_R$ makes $G$  into a skew brace by Theorem \ref{1.9}.

Suppose $L/K$ is a Galois extension with Galois group $(G, \circ) \cong G_L \times G_R$.  Since $(G, \circ, \cdot)$ is a skew brace, $L/K$ has a Hopf Galois structure by a $K$-Hopf algebra $H$ of type $(G, \cdot) = G_LG_R$.  Then the image of the Galois correspondence for  $H$ corresponds to the $\lambda_\circ(G)$-invariant subgroups of $(G, \cdot)$, and from Section 3 above, these are the $\circ$-stable subgroups of $(G, \cdot)$, namely, the subgroups $G'$ of $(G, \cdot)$ with the property that for all $h$ in $G'$, $g$ in $G$, the element $(g \circ h)\cdot g^{-1}$ is in $G'$.  

\begin{proposition}
Let $(G, \circ, \cdot)$ be the skew brace where $(G, \cdot) \cong G_LG_R$, a Zappa-Sz\'ep product  and let $\circ$ be the direct product operation on $G$ given by $g\circ x = g_L x g_R^{-1}$ for $g = g_Lg_R^{-1}$ in $G$. Then a  subgroup $G'$ of $(G, \cdot)$ is $\circ$-stable if and only if  $G'$ is normalized by $G_L$.
\end{proposition}

\begin{proof}  $G'$ is a $\circ$-stable subgroup of $(G, \cdot) \cong G_LG_R$ iff for all $x$ in $G'$, $g$ in $G$,   $( g \circ x)\cdot g^{-1} $ is in $G'$, and  
\[( g \circ x)\cdot g^{-1} = (g_Lxg_R^{-1})\cdot (g_Rg_L^{-1}) = g_Lxg_L^{-1}.\]
\end{proof}

\begin{example}   Let $(G, \cdot) = A_5 = C_5 \cdot A_4$ where $C_5 = \langle (1,2,3,4,5)\rangle$ and $A_4 = \Perm(\{1,2, 3,4\})$ is the stabilizer of 5 (c.f. [By15], Example 7.4).  Let $\sigma = (1,2,3,4,5)$.   A subgroup $G'$ of $A_5$ is $\circ$-stable iff $G'$ is normalized by $G_L$.   Three obvious $\circ$-stable subgroups are $\{(1)\}$, $C_5$ and $A_5$.  Claim:  there is exactly one other $\circ$-stable subgroup.

Since $\sigma^{-1}(a_1, \ldots, a_r)\sigma = (a_1 +1, \ldots, a_r + 1)$,  any non-trivial subgroup $G'$ of $A_5$ that is normalized by $\sigma$ is transitive, so has order a multiple of 5.  Since $A_5$ has no non-trivial subgroups of order $> 12$, $G'$ must have order 5 or 10, and has a characteristic subgroup $H$ of order 5, normalized by $\sigma$.  Thus $H = \langle \sigma \rangle$ and $G' = \langle \sigma, \rho \rangle$ where $\rho$ has order 2 and $\rho \sigma \rho^{-1} = \sigma^{-1}$.  (We can choose $\rho = (1,2)(3, 5)$.)

So there are four $\circ$-stable subgroups of $A_5$.  If $L/K$ is a Galois extension with Galois group $G = C_5 \times A_4$ with a Hopf Galois structure by $H$ of type $A_5$ corresponding to the skew brace defined by the fixed point free pair of homomorphisms from $G$ to $A_5$ as above, then of the intermediate fields $E$ with $K \subseteq E \subseteq L$, exactly four are in the image of the Galois correspondence for $H$.  

How many intermediate subfields are there?  How many subgroups are there of $C_5 \times A_4$?  Since $(|C_5|, |A_4|) = (12, 5) = 1$, the answer is:  twice the number of subgroups of $A_4$, hence $2 \cdot 10 = 20$ subgroups of $C_5 \times A_4$.  (See [DF99], page 112,  for the lattice diagram of subgroups of $A_4$.) So the proportion of intermediate subfields of $L/K$ that are in the image of the Galois correspondence for $H$ is
\[ GC((G, \circ), (G, \cdot)) =  
\frac{|\{\text{$\circ$-stable subgroups of $A_5$}\}|}{|\{\text{subgroups of $C_5 \times A_4$}|} =\frac 4{20}.\]
\end{example}

\section{Semi-direct products}

One set of examples of Zappa-Sz\'ep products are semidirect products of groups.  

Let $G = G_L \rtimes G_R$ be a semidirect product of two finite groups $G_L$ and $G_R$, where $G_L$ is normal in $G$ and the action of $G_R$ on $G_L$ is by conjugation.   

Denote the group operation in $G$ by $\cdot$, which we will often omit.  Thus for $x, y$ in $G$, $xy = x \cdot y$.   

An element of $G$ has a unique decomposition as $x = x_L x_R^{-1}$ for  $x_L$ in $G_L$, $x_R$ in $G_R$.    In  the semidirect product, an element $y_R$ of $G_R$ acts on $x_L$ in $G_L$ by conjugation:
\[ y_R^{-1} x_L = (y_R^{-1}x_Ly_R)y_R^{-1}.\] 
 Then 
\beal x y &= x_L x_R^{-1} y_L y_R^{-1} \\
& = x_L (x_R^{-1} y_Lx_R)x_R^{-1} y_R^{-1} 
.\eeal
 
Along with the given group operation on $G$ we also have  the direct product operation $\circ$ on $G$, as follows:
\beal  x \circ y &=  x_L x_R^{-1} \circ y_Ly_R^{-1}\\
&= x_L y_L y_R^{-1}x_R^{-1}\\
&=  x_L y x_R^{-1}.\eeal

\begin{proposition}\label{9.1}  $(G, \circ, \cdot)$ is a bi-skew brace. \end{proposition}

This was proved in [Ch19].

So suppose $L'/K'$ is a Galois extension with Galois group $(G, \cdot) = G_L \rtimes G_R$.  Since $(G, \cdot, \circ)$ is a skew brace, $L'/K'$ has a Hopf Galois structure by a $K'$-Hopf algebra $H'$ of type $(G, \circ) \cong G_L \times G_R$.  Then the image of the Galois correspondence for  $H'$ corresponds to the $\lambda_\cdot(G)$-invariant subgroups of $(G, \circ)$, and these are the $\cdot$-stable subgroups of $(G, \circ)$, namely, the subgroups $G'$ of $(G, \circ)$ with the property that for all $x$ in $G'$, $g$ in $G$, the element $(g \cdot x)\circ \ol{g}$ is in $G'$. 
  
\begin{proposition}  Let $(G, \circ, \cdot)$ be the bi-skew brace where $(G, \cdot) \cong G_L \rtimes G_R$, a semidirect product where $G_R$ acts on $G_L$ by conjugation, and let $\circ$ be the direct product operation on $G$ given by $g\circ x = g_L x g_R^{-1}$ for $g = g_Lg_R^{-1}$ in $G$. Then a subgroup $G'$ of $(G, \circ)$ is $\cdot$-stable if and only if for every $x = x_Lx_R^{-1}$ in $G'$ and all $g$ in $G$, $gx_Lg^{-1}x_R^{-1}$ is in $G'$.  \end{proposition}

\begin{proof}.   A subgroup $G'$ of $(G, \circ)$ is $\cdot$-stable iff for all $x$ in $G'$, $g$ in $G$, the element $(g \cdot x)\circ \ol{g}$ is in $G'$.   Note that if $g = g_Lg_R^{-1}$, then $\ol{g} = g_L^{-1}g_R$.  So for $g, x, y$ in $G$, 
\beal (g \cdot x)\circ \ol{g} &=g_Lg_R^{-1}x_Lx_R^{-1} \circ \ol{g} \\
&= (g_Lg_R^{-1}x_Lg_R)(g_R^{-1} x_R^{-1}) \circ \ol{g} \\
&= (g_Lg_R^{-1}x_Lg_R)\ol{g}  (g_R^{-1} x_R^{-1}) \\
&= (g_Lg_R^{-1}x_Lg_R)(g_L^{-1}g_R)(g_R^{-1} x_R^{-1}) \\
&= (g_Lg_R^{-1})x_L(g_Rg_L^{-1})(g_Rg_R^{-1}) x_R^{-1}\\
&= gx_Lg^{-1}x_R^{-1}.\eeal
\end{proof}

The remainder of the paper is devoted to examples.  In the examples, the circle operation $\circ$ is $+$, the usual addition of modular arithmetic.
\begin{example}
Let $(G, +) = G_L \times G_R = \ZZ_9 \times \ZZ_6$, the direct product with the usual operation, $(r, s) + (r', s') = (r + r', s + s')$, and identify $(r, s)$ with $r\cdot 2^s$ in $(G, \cdot) = \ZZ_9 \rtimes_2 U_9 \cong \Hol(C_9)$.  So $(r, s) \cdot (r', s') = (r + 2^sr', s + s')$.  

We wish to find the $+$-stable subgroups of $(G, \cdot) $ and the $\cdot$-stable subgroups of $(G, +)$.  Since a $+$-stable subgroup of $(G, \cdot)$ is a subgroup of $(G, +)$, and a $\cdot$-stable subgroup of $(G, +)$ is a subgroup of $(G, \cdot)$, all subgroups of interest are subgroups of the abelian group $(G, +)$.  So we begin by finding the subgroups of the direct product $(G, +)$, then see which are $+$-stable and which are $\cdot$-stable.

We find that there are 20 subgroups of $(G, +) = \ZZ_9 \times \ZZ_6$:  sixteen are cyclic groups, with generators:
\begin{center}\begin{tabular}
{c|c|c|c|c}
generator & order && generator & order\\\hline
(0,0) & 1 && (0,3) &2  \\
(1,0) &9  && (1,3) &18  \\
(3,0) &3  && (3,3) & 6 \\
(0,2) & 3 && (0,1) & 6 \\
(1,2) & 9 && (1,5) & 18 \\
(3,2) & 3 && (3,5) & 6 \\
(1,4) & 9 && (1,1) & 18 \\
(3,4) & 3 && (3,1) & 6 .\\
\end{tabular}\end{center}

The non-cyclic subgroups of $(G, +)$ are 
\beal   \langle (3, 0), (0, 2) \rangle_+ =  \langle (3, 0), (0, 2) \rangle_\cdot &\text{  of order  } 9,\\
 \langle (3, 0), (0, 1) \rangle_+ = \langle (3, 0), (0, 1) \rangle_\cdot &\text{  of order  } 18,\\
 \langle (1, 0), (0, 2) \rangle_+ =\langle (1, 0), (0, 2) \rangle_\cdot &\text{  of order  } 9,\\
  G = \langle (1, 0), (0, 1) \rangle_+ =\langle (1, 0), (0, 1) \rangle_\cdot &\text{  of order  } 54.\eeal
All are subgroups of $(G, \cdot)$.

A subgroup $G'$ of $(G, +)$ is $\cdot$-stable if for all $(r, s)$ in $G'$  and $(t, -u)$ in $G$, 
\[(t, -u)^{-1}(r, 0)( t, -u)(0, s)\]
is in $G'$.  Since $ (-2^{u}t, u)(t, -u)= (-2^ut + 2^ut, u + (-u)) = (0, 0)$, we have that
\beal (t, -u)^{-1}(r, 0)( t, -u)(0, s) &=  (-2^{u}t, u)(r, 1)( t, -u)(0, s)\\
&=  (-2^u t + 2^u r, u)( t, -u)(0, s)\\
 &=  (-2^u t + 2^u r + 2^ut, u -u)(0, s)\\
&=  ( 2^u r, 0)(0, s)\\  
  &= (2^ur, s)\eeal
 is in $G'$.   Setting $u = 1$  implies that $(2r, s)$ is in $G'$.  Since $G'$ is a  subgroup of $(G, +)$, therefore 
\[ (2r, s)-(r, s) = (r, 0)\]
 is in $G'$.  Thus:
 
 A subgroup $G'$ of $(G, \cdot)$ is $\cdot$-stable iff  for all $(r, s)$ in $G'$, $(r, 0)$ is in $G'$. 
 
Using this criterion, we find that the  four non-cyclic subgroups of $(G, +)$ are $\cdot$-stable,  and also the cyclic groups with generators
\[  (0 , 0),(1 ,0 ),(3 , 0), ( 0, 1), (0 , 2), (0 , 3), (3 , 3) \text{  and  }(1, 3),    \]
the last two because $10(1, 3) = (1, 0)$ and $4(3, 3) = (3, 0)$.  
Thus there are 12 $\cdot$-stable subgroups of $(G, +)$.

There are 32 subgroups of $(G, \cdot) = \ZZ_9 \rtimes \ZZ_6$, as follows:
There are 26 cyclic subgroups:  $\langle 0, 0\rangle $ of order 1; $\langle r, 3 \rangle $ of order  3 for $0 \le r < 9$;  $\langle 0, 2 \rangle,\langle 3, 0\rangle, \langle 3, 2\rangle, \langle 3, 4\rangle  $ of order 3;  $\langle r, 1\rangle$ of order  6 for $0 \le r < 9$;  and $\langle 1, 2s \rangle$ of order 3 for $0 \le s < 3$.  There are six non-cyclic subgroups, $\langle (a, 0), (0, b)\rangle$ for $a = 1, 3$ and $b = 1, 2, 3$.  

So  if  $L/K$ is $G$-Galois with $G \cong \ZZ_9 \rtimes \ZZ_6 \cong (G, \cdot)$, then for the Hopf Galois structure on $L/K$ corresponding to $(G, +) \cong \ZZ_9 \times \ZZ_6$, the ratio  
\[ GC((G, \cdot), (G, +)) 
= \frac{|\{\text{$\cdot$-stable subgroups of  }(G, +)\}|}{|\{\text{ subgroups of }(G, \cdot)\}|} = \frac {12}{32}.\]

  Now we look at $+$-stable subgroups of $(G, \cdot)$.
  
 A subgroup $G'$ of $(G, \cdot)$ is $+$-stable if $G'$ is a subgroup of both $(G, +)$ and $(G, \cdot)$ and is normalized by $G_L$ in $(G, \cdot)$, that is, 
$ (-t, 0)(r, s)(t, 0) $ is in $G'$ for all $t$ and all $(r, s)$ in $G'$.
Now
\beal (-t, 0)(r, s)(t, 0)  &= (-t + r, s)(t, 0)\\
&= (-t + r + 2^st, s)\\
&= ((2^s-1)t + r, s).\eeal
Since $G'$ is a group under $+$, $G'$ contains 
\[  ((2^s-1)t + r, s) - (r, s) = ((2^s-1)t, 0).\]
Setting $t = 1$, we have:

\[ \text{$G'$ is $+$-stable iff for all $(r, s)$ in $G'$, $(2^s -1, 0)$ is in $G'$. } \]

Thus, if $s = 1$, then $(1, 0)$ is in $G'$; if $s = 2$ then $(3, 0)$ is  in $G'$; if $s = 3$ then $(7, 0)$, hence $(1, 0)$ is in $G'$;  if $s = 4$ then $(15, 0)$, hence $(3, 0)$ is in $G'$; and if $s = 5$ then $(31, 0)$, hence $(1, 0)$ is in $G'$ .  

Among the cyclic subgroups of $(G, \cdot)$, this condition holds trivially for those with generators 
$(0, 0), (1, 0)$ and $(3, 0)$, and also for $(1, 2)$ and $(1, 4)$ since $(3, 0) = (1, 2) \circ (1, 2) \circ (1, 2) = (1, 4)(1, 4)(1, 4)$.   But the condition fails for all other cyclic subgroups of $(G, \cdot)$.     

The condition also holds for the non-cyclic subgroups 
\[ \langle (1, 0), (3, 0)\rangle, \langle (3, 0), (0, 2) \rangle,  \langle (1, 0), (0, 2) \rangle \text{ and }G,\]
 but not for $H =  \langle (3, 0), (0, 1) \rangle$ or $\langle (3, 0),(0, 3)\rangle$ because $(1, 0)$ is not in $H$.
 
 Thus there are nine $+$-stable subgroups of $(G, \cdot)$.  There are 20 subgroups of $(G, +)$.  So if  $L'/K'$ is $G$-Galois with $G \cong \ZZ_9 \times \ZZ_6 \cong (G, +)$, then for the $H'$-Hopf Galois structure  corresponding to $(G, \cdot) \cong \ZZ_9 \rtimes \ZZ_6$, the ratio  
\[GC( (G, +), (G, \cdot)) = \frac{|\text{$+$-stable subgroups of }(G, \cdot)|}{|\text{ subgroups of }(G, +)|} = \frac 9{20}.\]

 \end{example}

 \section{Groups of squarefree order}
Hopf Galois structures on groups of squarefree order were studied by [AB18], whose results included showing that if the field extension $L/K$ has a Galois group $G$ cyclic of squarefree order $mn$, then $L/K$ has a  Hopf Galois structure of type $N$ for every group $N$ of order $mn$: each such group $N$ must be a semidirect product of cyclic groups.

We look at some of those examples.

 Let $(G, +) = \ZZ_m \times \ZZ_n$ under componentwise addition, where $m$ and $n>1$ are coprime and squarefree and $n$ divides $\phi(m)$.  Then $(G, +)$ is cyclic of order $mn$, and every element of $G$ may be written as $(r, s) = (r, 0) + (0, s)$ for $r$ modulo $m$, $s$ modulo $n$. Also, $\langle (r, s)\rangle$ contains $(r, 0)$ and $(0, s)$ because $m$ and $n$ are coprime.  

The subgroups of $(G, +)$ are generated by $(r, s)$ where $r $ divides $m$ and $s$ divides $n$, so there are $d(m)d(n)$ subgroups of $(G, +)$, where $d(m)$ is the number of divisors of $m$.   If $m$ is a product of  $g$ distinct primes, and  $n$ is a product of $h$ distinct primes, then $d(m) = 2^g$, $d(n) = 2^h$.  Hence the number of subgroups of $(G, +)$ is $2^{g+h}$.  

Let $b$ have order $n$ in $U_m$, the group of  units modulo $m$.  Form the semidirect product $(G, \cdot) = \ZZ_m \rtimes_b \ZZ_n$ with the operation
 \[ (r, s) \cdot (r', s') = (r + b^sr', s + s').\]
 Then $(G, +, \cdot)$ is a bi-skew brace.
 
 We observe that every subgroup of $(G, +)$ is also a subgroup of $(G, \cdot)$.  For let $G' =\langle (r, s) \rangle =\langle (r, 0), (0, s) \rangle$ be a subgroup of $(G, +)$.  For elements $(cr, ds), (er, fs)$ of $G'$, 
 \[ (cr, ds) \cdot (er, fs) = (cr + b^{ds}er, (d + f)s) = ((c + b^{ds})r,  (d + f)s),\]
 which is in $\langle (r, 0), (0, s) \rangle$.  

Since the $+$-stable subgroups  and the $\cdot$-stable subgroups of $(G, +, \cdot)$ are subgroups of both $(G, +)$ and $(G, \cdot)$, we may  search for each  from among the subgroups $\langle (r, s) \rangle$ of the cyclic group $(G, +)$, where $r$ divides $m$ and $s$ divides $n$.

We first find the $\cdot$-stable subgroups of $(G,  +)$.  

 \begin{proposition}\label{8.1}  Let $(G, +) = \ZZ_m \times \ZZ_n \cong \ZZ_{mn}$ with $(m, n) = 1$, and $(G, \cdot) = \ZZ_m \rtimes_b \ZZ_n$ with $b$  of order $n$ modulo $m$.  Then every subgroup of $(G, +)$ is $\cdot$-stable. \end{proposition}
 
\begin{proof}  A subgroup $G'$ of $(G, +)$ is $\cdot$-stable if and only if for all $(x, y)$ in $G'$ and all $g$ in $G$, $(g \cdot x \cdot g^{-1}, y)$ is in $G'$ (where $g^{-1}$ is the $\cdot$ inverse of $G$).

Let $(a, -h)$ be in $G$, then $(a, -h)^{-1} = (-b^ha, h)$.  So for $(x, y) = (x, 0)+ (0, y)$ in $G'$, 
\beal  (-b^ha, h)(x,0)(a, -h)+ (0, y)
&= (-b^ha + b^hx, h)(a, -h)+(0, y)\\
&= (-b^ha + b^hx + b^ha, 0)+(0,y)\\
 &= (b^hx, 0)+(0, y). \eeal 
This is in $G'$ because $(b^hx, 0)$ and $(0, y)$  are in $G'$.

Thus every subgroup of $(G, +)$ is $\cdot$-stable.
\end{proof}

It follows that for a Galois extension $L/K$ with Galois group isomorphic to $(G, \cdot) \cong \ZZ_m \rtimes \ZZ_n$ with a Hopf Galois extension of type  $(G, +) \cong \ZZ_{mn}$  corresponding to the skew brace $(G, \cdot  , + )$, the proportion of intermediate subfields of $L/K$ that are in the image of the Galois correspondence for the Hopf Galois structure on $L/K$ is 
\[GC((G, \cdot), (G, +)) = \frac{|\{\text{   subgroups of }\ZZ_{mn}\}|}{|\{\text{ subgroups of } \ZZ_m \rtimes \ZZ_n \}|}.\]

 Now we look at the $+$-stable subgroups of $(G, \cdot)$.  
 
 A subgroup $\langle (r, 0), (0, s) \rangle$ of $(G, \cdot) =\ZZ_m \rtimes Z_n$ is $+$-stable iff $(b^s -1, 0)$ is in $\langle  (r, 0)\rangle$, iff $r$ divides $b^s-1$.  

Now  $m = p_1\cdots p_g$ and  $n = q_1 \cdots q_h$, products of distinct primes.  So if $r = p_{i_1} \cdots p_{i_k}$ and $b$ has order $n_i$ modulo $p_i$, then $(b^s  -1, 0)$ is in $\langle  (r, 0)\rangle$ if and only if the least common multiple $[n_{i_1}, \ldots, n_{i_k}]$ divides $s$. 

\begin{example}\label{factor}  Consider $(G, +) = \ZZ_{pq} \cong \ZZ_p \times \ZZ_q$ where $q$ is a prime divisor of $p-1$. Let $b$ have order $q$ modulo $p-1$ and let $(G, \cdot) = \ZZ_p \rtimes_b \ZZ_q$ with operation
\[ (r_1, s_1)(r_2, s_2) = (r_1 + b^{s_1}r_2, s_1 + s_2).\]
Then $(G, +)$ has four subgroups, generated by $(1, 0), (0, 1), (1, 1)$ and $(0, 0)$, of orders $p, q, pq$ and 1, respectively.
A subgroup $G' = \langle (r, s) \rangle$ of $(G, \cdot)$ is $+$-stable if and only if it is a subgroup of $(G, +)$ and   $b^s \equiv 1 \pmod{r}$.  The only subgroup of $(G, +)$ that is not $+$-stable is $(0, 1)$, because $b-1$ is not in $\langle 0 \rangle \subset \ZZ_p$.   So for a $(G, +)$-Galois extension $L/K$ with an $H$-Hopf Galois extension of type $(G, \cdot)$,  the ratio, 
\[ GC((G, +), (G, \cdot)) = \frac{|\{ +\text{-stable subgroups of } (G, \cdot)\}|}{|\{\text{subgroups of }(G, +)\}|}  = \frac 34.\]

Every subgroup of $(G, +)$ is $\cdot$-stable.   A computation shows that the number of subgroups of $(G, \cdot)$ is $p+3$:  the $p$ cyclic subgroups of order $q$ generated by $(r, 1)$ for $r = 0, \ldots, p-1$, together with the group generated by $(1, 0)$, of order $p$, and the two trivial groups, $\{0\}$ and $G$.  So for a $(G, \cdot)$-Galois extension $L'/K'$ with an $H'$-Hopf Galois extension of type $(G, +)$, the ratio 
\[ GC( (G, \cdot), (G, +)) =\frac{|\{ \cdot\text{-stable subgroups of } (G, +)\}|}{|\{\text{subgroups of }(G, \cdot) \}|} = \frac 4{p+3}.\]
\end{example}

\begin{example}  Now let $(G, +) = \ZZ_{mn}$ where $m = p_1 \cdots p_g$, and $n = q_1 \cdots q_g$, all pairwise distinct primes, where for $i = 1, \ldots g$, $q_i$ divides $p_i-1$.  Let $b$ have order $q_i$ modulo $p_i$ for all $i$, so $b$ has order $n$ modulo $m$.   Then 

\begin{proposition} With $m, n, b$ chosen as above, 
\beal (G, +)&= \ZZ_m \times \ZZ_n \cong \ZZ_{p_1q_1} \times \cdots \times  \ZZ_{p_gq_g},\\  (G, \cdot)&= \ZZ_m \rtimes_b \ZZ_n \cong \ZZ_{p_1} \rtimes_b \ZZ_{q_1}\times \cdots \times  \ZZ_{p_g} \rtimes_b \ZZ_{q_g},\eeal
and every subgroup of $(G, +)$, resp. $(G, \cdot)$ is a direct product of its projections onto the corresponding subgroups.  \end{proposition}

The decomposition of $(G, +)$ is obvious; that of $(G, \cdot)$  is a routine induction argument from the case $g = 2$, which in turn follows because the order of $b$ modulo $p_1p_2$ is the product of the coprime orders of $b$ modulo $p_1$ and modulo $p_2$.  The statements about subgroups follow from Goursat's Lemma and the fact that the direct factors have coprime order (or by a Chinese Remainder Theorem argument).

Thus the ratios of Proposition \ref{factor} multiply to yield

\begin{corollary}  With $(G, +)$ and $(G, \cdot)$ as above, 
\beal  GC( (G, \cdot), (G, +)) &= \frac{|\{ \cdot\text{-stable subgroups of } (G, +)\}|}{|\{\text{subgroups of }(G, \cdot)\}|}\\&  = \frac {4^g}{(p_1 +3)(p_2 +3)\cdots( p_g+3)},\eeal and 
\[ GC((G, +), (G, \cdot)) = \frac{ |\{ +\text{-stable subgroups of } (G, \cdot)\}|}{|\{\text{subgroups of }(G, +)\}| } = (\frac 34)^g.\]    \end{corollary}

Both ratios go to zero for large $g$.

\end{example}

For a final example, we look at a generalization of the dihedral group $D_m$ where $m$ is odd and squarefree.
\begin{example}
Let $m = p_1 \cdots p_g$, a product of distinct primes, and let $n = q_1\cdots q_h$ where $q_1, \ldots q_h$ are  distinct primes that divide $p-1$ for every prime $p$ dividing $m$. (The dihedral case is  $h =1, n = 2$).  Let $b$ have order $n$ modulo $p_i$ for every $i$.  Let $(G, +) = \ZZ_m \times \ZZ_n \cong \ZZ_{mn}$, $(G, \cdot) = \ZZ_m \rtimes_b \ZZ_n$.  Then the  subgroups  of $(G, +)$ all have the form
$ \langle (r, s) \rangle = \langle (r, 0), (0, s)\rangle$ where $r$ divides  $m$, $s$ divides $n$.

Since every $+$-stable subgroup $G'$ of $(G, \cdot)$ is a subgroup of $(G, +)$, we can assume 
$G' = \langle  (r, s) \rangle$ where $r$ divides $m$, $s$ divides $n$.  Now $G'$ is $+$-stable iff $b^s -1$ is in $\langle r \rangle$.  
 We have two cases:

Case 1:  $r = 1$ and  $b^s -1$ is in $\langle 1 \rangle$ for all $s$.  

Case 2:   $1 < r$ and $r$ divides $m$.   Then some $p_i$ divides $r$, so $b$ has order $n$ modulo $p_i$.  
Then $b^s \equiv 1 \pmod{r}$ if and only if $s \equiv 0 \pmod{n}$.

Thus $\langle (r, s)\rangle$ is $+$-stable for $r = 1$ and all $s$, or for $ r \ne 1$ and $s = n$.  Since we may assume that $r$ divides $m$, $s$ divides $n$, then the number of $(r, s)$ in Case 1 is $2^h$, and the number of $(r, s)$ in Case 2 is $2^g -1$.  So the number of $+$-stable subgroups of $G$ is  $2^h + 2^g -1$.  

The number of subgroups of $(G, +)$ is $2^h \cdot 2^g$.  So the ratio 
\beal GC((G, +), (G, \cdot) )&=  \frac{|\{\text{  $+$-stable subgroups of $(G, \cdot)$}\}|}{|\{ \text{ subgroups of $(G, +)$}\}|}\\& = \frac{2^h  + 2^g -1}{2^{h+g}}. \eeal

On the other hand, by Proposition \ref{8.1}, every subgroup of $(G, +)$ is $\cdot$-stable, so the ratio
\beal GC((G, \cdot), (G, +))&= ( \frac{|\{\text{ $\cdot$-stable  subgroups of $(G, +)$}\}|}{|\{ \text{ subgroups of $(G, \cdot)$}\}|}\\
& = \frac{ |\{\text{ subgroups of $(G, +)$}\}|}{|\{\text{ subgroups of $(G, \cdot)$}\}|}.\eeal

Since $\ZZ_ m \rtimes \ZZ_n$ is metabelian, every subgroup of $(G, \cdot) = \ZZ_m \rtimes \ZZ_n$ has the form $H \rtimes K$ where $H < \ZZ_m$, $K < \ZZ_n$.   If $K = (0)$ then we're counting the $2^g$ subgroups of $\ZZ_m$.  For each $s$ dividing $n$ and $r$ dividing $m$ there are $r$ subgroups of order $(m/r)(n/s)$ of the form
$ \langle (r, 1), (t, s) \rangle$ for $0 \le t < r$.   The total number of subgroups of $(G, \cdot)$ is then 
\[ 2^g + \sum_{s|n, s \ne n} \sum_{r|m} r = 2^g + (2^h -1)\sum_{r|m} r = 2^g + (2^h-1)\sigma(m),\]
where $\sigma(m)$, the sum of the divisors of $m = p_1\cdot \ldots \cdot p^g$, is 
\[ \sigma(m) = \prod_{i = 1}^g (1 + p_i).\]
Thus 
\[ GC((G, \cdot), (G, +)) = \frac {2^{h+g}}{2^g + (2^h-1)\sigma(m)}.\]
Since $\sigma(m) \ge 3^g$,  
\[ GC((G, \cdot), (G, +)) \le  2(\frac 23)^g, \]
which is close to 0 for large $g$.

In particular, for $\ZZ_n = \ZZ_2$, the dihedral case with $m$ odd and squarefree, the ratio
\[ GC((G, \cdot)), (G, +) ) =  GC(D_m, \ZZ_{2m})  \le 2(\frac 23)^g \]
goes to 0 with $g$, while
\[ GC((G, +)), (G, \cdot) ) =GC(\ZZ_{2m}, D_m) = \frac{2^g +1}{2^{g+1}} > \frac 12 \]
for all $g$.

Responding to the question raised in Example \ref{8.2}, this example shows that given a bi-skew brace $(G, \cdot, +)$ with $(G, \cdot) \cong D_m$ and $(G, +) \cong \ZZ_{2m}$ for highly composite square-free odd $m$,   the ratios describing the images of the Galois correspondences for the Hopf Galois structures of type $\ZZ_{2m}$, resp. $D_m$, on Galois extensions with Galois group $D_m$, resp. $\ZZ_{2m}$, corresponding to the bi-skew brace could hardly be more dissimilar.
\end{example}


\begin{thebibliography}{[CDVS06]}

\bibitem
[AB18] {AB18} 
A. A. Alabdali, N. P. Byott, Counting Hopf-Galois structures on cyclic field extensions of squarefree degree, J. Algebra 493 (2018), 1--19.

\bibitem
[By15]{By15}
N. P. Byott, Solubility criteria for Hopf-Galois structures, New York J. Math. 21 (2015), 883-903.

\bibitem
[BC12]{BC12}
N. P. Byott, L. N. Childs, Fixed point free pairs of homomorphisms and nonabelian Hopf Galois structures, New York J. Math. 18 (2012), 707--731.

\bibitem  
[CDVS06]{CDVS06}
A. Caranti, F. Dalla Volta, M.  Sala, Abelian regular subgroups of the affine group and radical rings, Publ. Math. Debrecen 69 (2006),  297--308. 
MR2273982 , Zbl{1123.20002}.

\bibitem  
[CS69]{CS69}
S. U. Chase, M. E.  Sweedler, Hopf Algebras and Galois Theory, Lect. Notes in Math., 97 Springer-Verlag, Berlin-New York (1969)  ii+133 pp.  MR0260724,  Zbl 0197.01403.

\bibitem
[Ch17]{Ch17}
L. N. Childs, On the Galois correspondence for Hopf Galois structures, 
New York J. Math.  (2017), 1--10.

\bibitem
[Ch18]{Ch18}
L. N. Childs, Skew braces and the  Galois correspondence for Hopf Galois structures, 
 J. Algebra 511  (2018), 270--291. Zbl {1396.12003},  arXiv:1802.03448.
 
\bibitem
[Ch19]{Ch19}
L. N. Childs, Bi-skew braces and Hopf Galois structures, New York J. Math. 25 (2019), 574-588.

\bibitem
[CCo07]{CCo07}
L. N. Childs, J.  Corradino, Cayley's Theorem and Hopf Galois structures arising from semidirect products of cyclic groups, J. Algebra 308 (2007), 236--251.

\bibitem
[CG18]{CG18}
L. N. Childs,  C. Greither, Bounds on the number of ideals in finite commutative nilpotent $\Fp$-algebras, Publ. Math. Debrecen 92 (2018), 495--516.


\bibitem
[CRV16]{CRV16}
T. Crespo, A. Rio, M. Vela, On the Galois correspondence theorem in separable Hopf Galois theory, Publ. Math. (Barcelona) 60 (2016), 221--234.

\bibitem
[DeC19]{DeC19}
K. De Commer,  Actions of skew braces and set-theoretic solutions of the reflection equation, Proc. Edinburgh Math. Soc. 62 (2019), 1089-1113.

\bibitem
[DeG18]{DeG18} 
W. A. De Graaf,  Classification of nilpotent associative algebras of small dimension,  International Journal of Algebra and Computation 28 (2018), 133--161. MR3768261, Zbl {06850907}, arXiv:1009.5339.

\bibitem
[DF99]{DF99}
D. S. Dummit, R. M. Foote, Abstract Algebra, 2nd edition, John Wiley and Sons, New York, 1999.

\bibitem
[FCC12]{FCC12}  
S. C. Featherstonhaugh, A. Caranti,  L. N. Childs,  Abelian Hopf Galois structures on prime-power Galois field extensions,  Trans. Amer. Math. Soc. 364 (2012), 3675--3684.  MR2901229, Zbl 1287.12002.

\bibitem
[GP87]{GP87} 
 Greither,C.;  Pareigis, B.  Hopf Galois theory for separable field extensions, J. Algebra 106 (1987),  239--258. MR0878476, Zbl{0615.12026}.

\bibitem
[GV17]{GV17}  
L.  Guarnieri, L. Vendramin,  Skew braces and the Yang-Baxter equation, Math. Comp. 86 (2017), no. 307, 2519--2534. MR3647970, Zbl 1371.16037,  arXiv:1511.03171.

\bibitem
[He61]{He61}
I. N. Herstein,   Theory of Rings, University of Chicago Mathematics Lecture Notes, Spring, 1961.

\bibitem
[KT20]{KT20}
A. Koch, P. J. Truman, Opposite skew braces and applications, J. Algebra 546 (2020), 218--235.

\bibitem
[NZ19]{NZ19}
K.  Nejabati Zenouz,   Skew braces and Hopf-Galois structures of Heisenberg type,   J. Algebra 524 (2019), 187--225.  Zbl{07027115}.

\bibitem
[Ru07]{Ru07}
W. Rump,  Braces, radical rings, and the quantum Yang-Baxter equation,  J. Algebra 307 (2007), 153--170. MR2278047, Zbl {1115.16022},  arXiv:1804.01360.

\bibitem
[SV18] {SV18}
A. Smoktunowicz, L. Vendramin, On skew braces (with an appendix by N. Byott and L.  Vendramin), J. Comb. Algebra 2(2018), 47--86.






\end{thebibliography}
\end{document}